\documentclass[10pt]{article}
\usepackage[DIV=9,BCOR=0mm]{typearea}

\title{Dynamic Mode Decomposition: Theory and Data Reconstruction}
    \author{Tim Krake\footnote{University of Stuttgart and Hochschule der Medien, Germany (tim.krake@visus.uni-stuttgart.de)} , Daniel Weiskopf\footnote{University of Stuttgart, Germany (weiskopf@visus.uni-stuttgart.de)} , Bernhard Eberhardt\footnote{Hochschule der Medien, Germany (eberhardt@hdm-stuttgart.de)}
}
\date{\today}

\usepackage[T1]{fontenc}
\usepackage{lmodern}
\RequirePackage[english]{babel}

\usepackage{cite}
\usepackage{amsmath,amssymb}
\usepackage{mathtools}
\usepackage[amsmath,thmmarks]{ntheorem}

\newtheorem{definition}{Definition}[section]
\newtheorem{theorem}[definition]{Theorem}
\newtheorem{corollary}[definition]{Corollary} 
\newtheorem{lemma}[definition]{Lemma} 
\newtheorem{proposition}[definition]{Proposition}
\newtheorem{algorithm}[definition]{Algorithm}
\newtheorem{remark}[definition]{Remark}

\theoremstyle{nonumberplain}
\theoremseparator{.} 
\theoremheaderfont{\normalfont \bfseries}
\theorembodyfont{\normalfont} 
\theoremsymbol{\ensuremath{_\Box}} 
\newtheorem{proof}{Proof}		

\newcommand{\rank}{\textnormal{rank}}
\newcommand{\im}{\textnormal{im}}
\renewcommand{\ker}{\textnormal{ker}}
\renewcommand{\dim}{\textnormal{dim}}
\newcommand{\diag}{\textnormal{diag}}
\newcommand{\Vand}{\textnormal{Vand}}

\usepackage{hyperref}
\usepackage[capitalize,nameinlink,noabbrev]{cleveref}

\begin{document}

\maketitle

\begin{abstract} \noindent
Dynamic Mode Decomposition (DMD) is a data-driven decomposition technique extracting spatio-temporal patterns of time-dependent phenomena. 
In this paper, we perform a comprehensive theoretical analysis of various variants of DMD.
We provide a systematic advancement of these and examine the interrelations. 
In addition, several results of each variant are proven. Our main result is the exact reconstruction property.
To this end, a new modification of scaling factors is presented and a new concept of an error scaling is introduced to guarantee an error-free reconstruction of the data. 
\\

\end{abstract}

\section{Introduction}
The analysis of time-dependent phenomena is at the heart of investigation in a broad range of scientific research.
Within these studies, the integration of data in the form of time-series has increased considerably.
Therefore, the application of innovative algorithms is necessary to gain deep insights into the characteristics of data.
In this paper, we address time-series analysis by Dynamic Mode Decomposition (DMD), which was first introduced by Schmid and Sesterhenn in 2008.

DMD is a data-driven and model-free algorithm extracting spatio-temporal patterns in the form of so-called DMD modes and DMD eigenvalues.
As an efficient tool in fluid mechanics, DMD has gained much attention.
DMD has been investigated on both practical and theoretical grounds.
Nonetheless, the focus of these analyses was mainly a practical one.
For example, various types of flow were considered, such as airflow around an airfoil, fuel flow in a combustion chamber, or heat conduction in various cases. Completely different fields of application comprise financial trading, video processing, epidemiology, neuroscience, and control theory.

In contrast, we focus on theoretical investigations. The paper is thus structured as follows:
After discussing related work, we introduce the theoretical framework of DMD dealing with the background mechanisms.
In this process, we define the so-called system matrix, which is pioneering for DMD and prove the following results: a characterization for the exactness and diagonalizability of the system matrix as well as the resulting reconstruction of data with its spectral components.
These theorems are central for the following sections introducing the three common variants of DMD: The original formulation~\cite{schmid:2008:APS}, the modification by a singular value decomposition~\cite{schmid:2011:applications_of_dmd}, and Exact-Dynamic Mode Decomposition~\cite{tu:2014:on_dmd_theorey_and_app}.
In this context, a systematic advancement will be presented that clarifies precisely the interrelation of these algorithms.
This especially includes algebraic identities as well as spectral-theoretic results leading, e.g., to a new approach for the extension to the most recent variant of DMD.
In addition, the exact reconstruction property of DMD will be proven for each DMD variant that guarantees an error-free reconstruction of the data.
To this end, a new variant of scaling factors is introduced involving a new concept of an error scaling for the reconstruction of the first snapshot.
Some concluding remarks will be given in the last section.

\section{Related Work}
Rowley et al. provided a first theoretical investigation~\cite{rowley:2009:spectral_nonlinear_flow} for the fundamental version of DMD, here denoted as Companion Dynamic Mode Decomposition (CDMD).
They dealt with the reconstruction property of CDMD, however, they did not take appropriate scaling factors into account.
An algorithmic improvement through the singular value decomposition was achieved by Schmid resulting in another variant of DMD~\cite{schmid:2010:dmd_numerical_data}.
We refer to this algorithm as Singular Value Decomposition Dynamic Mode Decomposition (SDMD).
The reconstruction property of SDMD was mentioned by Chen et al. as well as further properties of CDMD~\cite{chen:2012:variants_of_DMD}.
Tu et al. introduced the advancement of SDMD to Exact-Dynamic Mode Decomposition (EXDMD)~\cite{tu:2014:on_dmd_theorey_and_app}, which is the most recent version of DMD.
They prove basic algebraic identities and show primarily spectral-theoretic connections between these two algorithms.
We generalize and extend all results or derive them as a corollary.
In addition, we present a new approach for the extension of SDMD to EXDMD characterizing precisely the connection between these two algorithms. 
Despite these theoretical investigations, the problem of an exact reconstruction is left open.
However, Jovanovic et al.~\cite{jovanovic:2014:sparsity_promoting_dmd} as well as Drma{\v{c}} et al.~\cite{drmavc:2018:data_driven_koopman_spectral_analysis_in_vandermonde_cauchy_form} discussed efficient techniques for finding appropriate coefficients by solving certain (convex) minimization problems.
We introduce a new variant of scaling factors that lead to an error-free reconstruction of the snapshots under appropriate conditions.

\section{Theoretical Framework}
This section is dedicated to the basic theoretical background of DMD.
In this context, the general setting will be presented as well as an intuitive interpretation of the principles of DMD.
These are crucial for the precise understanding of DMD, forming the basis for the subsequent sections. 
In addition, basic notation will be formalized and consistently used in this paper.

The application of DMD starts with the availability of data that may stem either from empirical experiments or numerical simulations alike.
The objective of DMD is to extract spatio-temporal patterns out of the data in the form of DMD modes, eigenvalues ,and amplitudes.
As the modes are related to spatial structures, the corresponding eigenvalues determine the temporal behavior of these.
The amplitudes characterize the impact of individual modes on the whole system, i.e. the dominance structure. 

Now, consider data (snapshots) $x_0,x_1,\dots,x_m \in \mathbb{C}^n$ with the following two quantities:
\begin{align*}
n &= \text{``size/dimension of the data points''}, \\
m &= \text{``number of data points''}.
\end{align*}
In the application areas of DMD such as fluid dynamics or non-linear dynamics, the connection between these variables is typically given by $n \gg m$, which means that the size of the data points is considerable larger than the number of snapshots.
In this context, typical values are $n \approx 10^6$--$10^{12}$ (depending on whether we address 2D or 3D scenarios) and $m \approx 100$--$1000$.
This basic setting is crucial for understanding the principles of DMD and will be assumed in the following derivation.
However, DMD can also be mathematically formulated and applied without this assumption, as we will see later.

In short, DMD calculates the relevant dynamic information of a high-dimensional linear operator that connects the given data points $x_0,x_1,\dots,x_m $ in a least square sense, without explicitly computing it.
This is achieved by an eigenvalue decomposition of a low-dimensional representation.
The corresponding eigenvectors will be embedded as DMD modes into the high-dimensional space endowed with appropriate scaling factors, the DMD amplitudes.

In order to obtain the high-dimensional matrix $A \in \mathbb{C}^{n \times n}$ connecting the data points, we consider the following (least-squares) minimization problem:
\begin{equation*}
\min_{A \in \mathbb{C}^{n \times n}} \sum_{j=0}^{m-1} \lVert A x_j - x_{j+1} \rVert_2^2.
\end{equation*}
Note that the high dimensionality stems from the fact $n \gg m$.
An explicit solution of $A$ is necessary to formulate an algorithmic approach.
To this end, we rewrite the data into the matrices
\begin{equation*}
X = \begin{bmatrix}
|	& |		& 		& |			\\ 
x_0 &  x_1 	& \dots & x_{m-1}	\\
|	& |		& 		& |			\\ 
\end{bmatrix}
,
\quad
Y = \begin{bmatrix}
|	& |		& 		& |			\\ 
x_1 &  x_2 	& \dots & x_{m}		\\
|	& |		& 		& |			\\ 
\end{bmatrix} \in \mathbb{C}^{n \times m}
\end{equation*}
obtaining the following equivalent minimization problem:
\begin{equation*}
\min_{A \in \mathbb{C}^{n \times n}} \lVert A X - Y \rVert_F^2,
\end{equation*}
where $\lVert \cdot \rVert_F$ denotes the Frobenius norm of a matrix. An explicit solution is now given by
\begin{equation*}
A = Y X^+ \in \mathbb{C}^{n \times n},
\end{equation*}
where $X^+$ denotes the Moore-Penrose pseudoinverse~\cite{moore:1920:moore_penrose} of $X$. 
Since the pseudoinverse always exists, the solution $A$ can be used for an algorithmic formulation.

Now, assuming the diagonalizability of the matrix $A$, i.e., $A = V \Lambda V^{-1}$ with the matrices $\Lambda = \diag(\lambda_1,\dots,\lambda_n)$ and $V = \begin{bmatrix} v_1 & v_2 & \dots & v_n \end{bmatrix}$ containing the eigenvalues and eigenvectors respectively, we obtain the characteristic reconstruction property of the matrix $A$ by
\begin{equation*}
x_k 
\approx A^k x_0 
= (V \Lambda V^{-1})^k x_0
= V \Lambda^k V^{-1} x_0
= V \Lambda^k b
= \sum_{j=1}^n \lambda_j^k b_j v_j
\end{equation*}
for $k = 0,1,\dots,m$, where $b=(b_1,\dots,b_n)^T$ are the coefficients of the linear combination of $x_0$ in the eigenvector basis, i.e. $b = V^{-1} x_0$.
Since the rank of $A$ is at most $\min\{\rank(X),\rank(Y)\}$ and consequently not more than $m$, there are at least $n-m$ eigenvalues of $A$ that are equal to zero.
The dynamic behavior will be thus captured by at most $m$ components, which are considerable fewer components.
Consequently, we obtain the following reconstruction of the data:
\begin{equation*} 
x_0 \approx \sum_{j=1}^m b_j v_j + q_0
\qquad
x_k \approx \sum_{j=1}^m \lambda_j^k b_j v_j,
\end{equation*}
for $k = 1,\dots,m$, where $q_0$ is the resulting error arising from the missing $m-n$ components.
In sum, we gain a reasonable low-dimensional decomposition of the data into the triples $(\lambda_j, v_j, b_j) \in \mathbb{C} \times \mathbb{C}^n \times \mathbb{C}$, providing an instrument for diagnostic approaches as well as a tool for prediction, long-term analysis, and stability analysis. 

The different versions of DMD presented in the subsequent sections are based on various techniques to produce a low-dimensional representation of the matrix $A$ in order to (approximately) compute its eigenvalues and eigenvectors as well as new appropriate scaling factors. 
These procedures yield similar triples that will be denoted by $(\lambda_j, \vartheta_j, a_j) \in \mathbb{C} \times \mathbb{C}^n \times \mathbb{C}$ throughout the paper. 
These triples consist of the so-called DMD eigenvalues, DMD modes, and DMD amplitudes corresponding to a particular algorithm (see each section).

Before studying the variants of DMD, we first concentrate on an analysis of the high-dimensional structures involving the matrix $A$.
Through a deeper understanding of the matrix $A$ representing the starting point of DMD, we obtain insights into the desired action of DMD. 
In particular, the success of an error-free reconstruction of DMD depends on the following two aspects:
\begin{enumerate}
\item The exactness of the matrix $A$, i.e., $A X = Y$.
\item The diagonalizability of the matrix $A$.
\end{enumerate}
These two aspects will be examined throughout this section, however, before, the matrix $A$ will be captured in the following definition.

\begin{definition} For data $x_0,\dots,x_m \in \mathbb{C}^n$ with associated matrices $X = \begin{bmatrix} x_0 & \dots & x_{m-1} \end{bmatrix}$ and $Y = \begin{bmatrix} x_1 & \dots & x_m \end{bmatrix}$, we call the matrix $A = Y X^+ \in \mathbb{C}^{n \times n}$ the system matrix to the data $x_0,\dots,x_m$.
\end{definition}

First, we recall some well-known facts. 
The definition is well-defined, as the pseudoinverse always exists and is unique.
Furthermore, the system matrix is the unique solution to the the minimization problem $\min_{A \in \mathbb{C}^{n \times n}} \lVert A X - Y \rVert _F^2$, if the rows of the matrix $X$ are linearly independent (or equivalently, if the matrix $X$ is surjective).

An important condition is the exactness of the system matrix, i.e., the equality $AX =Y$ or equivalently $A x_j = x_{j+1}$ for $j = 0,\dots,m-1$. The following proposition characterizes this property on linear functionals.

\begin{proposition}
Let the system matrix $A$ be given to the data $x_0,\dots,x_m \in \mathbb{C}^n$ as well as an arbitrary vector $w \in \mathbb{C}^{n}$. Then the following statements are equivalent:
\begin{enumerate}
\item[(i)] $\ker(X) \subseteq \ker(w^*Y)$, i.e., $Xz = 0 \implies w^*Yz = 0$ for all $z \in \mathbb{C}^n$.
\item[(ii)] The system matrix is exact on $w$, i.e. $w^*A X=w^*Y$.
\end{enumerate}
\end{proposition}

\begin{proof} ``$(i) \implies (ii)$''. Consider the following equation
\begin{equation*}
w^*Y-w^*A X = w^*Y - w^*Y X^+ X = w^*Y( I - P_{X^*}),
\end{equation*}
where $P_{X^*}$ is the orthogonal projection onto the image of $X^*$. 
Therefore $I - P_{X^*}$ is the orthogonal projection onto the kernel of $X$ and consequently the assertion follows by the assumption $\ker(X) \subseteq \ker(w^*Y)$.

``$(ii) \implies (i)$''. Let $z \in \ker(X)$, i.e., $Xz = 0$. This implies $w^* A X z = 0$, which is by assumption equivalent to $w^* Y z = 0$. Hence $z \in \ker(w^*Y)$.
\end{proof}

A simple consequence of this proposition is the following corollary~\cite[Theorem 2]{tu:2014:on_dmd_theorey_and_app}.

\begin{corollary} \label{2:Corollary_AX=Y}
Let the system matrix $A$ be given to the data $x_0,\dots,x_m \in \mathbb{C}^n$. Then the following assertions are equivalent:
\begin{enumerate}
\item[(i)] $\ker(X) \subseteq \ker(Y)$.
\item[(ii)] The system matrix is exact, i.e. $A X=Y$.
\end{enumerate}
\end{corollary}

In the case of linear independent data points $x_0,\dots,x_{m-1}$, the system matrix is exact, since $\ker(X) = \{0\}$ and hence condition $(i)$ of \cref{2:Corollary_AX=Y} is trivially satisfied.
In this context, the condition $n \gg m$ (which is typical for the application areas of DMD) suggests the linear independence of the data.
Consequently, the first aspect (exactness of the system matrix) is characterized.
For a full reconstruction of the data, however, we still need the diagonalizability of the system matrix.
To this end, we examine the inner structure of the system matrix, i.e., its kernel and image. 

\begin{lemma} \label{2:Lemma_DimImKerA}
Let the system matrix $A$ be given to the data $x_0,\dots,x_m \in \mathbb{C}^n$, where $\allowbreak x_0, \allowbreak \dots,\allowbreak x_{m-1}$ and $x_1,\dots,x_m$ are linear independent, respectively. Then the dimension of the image and the kernel of $A$ is given by
\begin{equation*}
\dim~\im(A)=m, 
\qquad
\dim~\ker(A)=n-m.
\end{equation*}
\end{lemma}
\begin{proof}
For the first assertion, we use the following rank inequality:
\begin{equation*}
m  = \rank(Y) + \rank(X^+) - m \leq \rank(Y X^+) \leq \min\{\rank(Y),\rank(X^+)\} = m,
\end{equation*}
which implies $\rank(A) = m$ or equivalently $\dim~\im (A) = m$.
Consequently, we obtain that $\dim~\ker (A) = n - \dim~\im (A) = n-m$. 
\end{proof}

\begin{corollary} \label{2:Corollary_DiagoA}
Let the system matrix $A$ to the data $x_0,\dots,x_m \in \mathbb{C}^n$ be given, where $x_0, \dots, x_{m-1}$ and $x_1,\dots,x_m$ are linear independent, respectively. If $A$ has $m$ non-zero distinct eigenvalues, then it is diagonalizable.
\end{corollary}

Regarding the reconstruction property of the system matrix, the final result will be stated in the following theorem, which deals with both sufficient and necessary conditions.
For a simple notational handling of the proof, we define the Vandermonde matrix.

\begin{definition} For $\lambda_1,\dots,\lambda_m \in \mathbb{C}$ we define the $k$-$K$-Vandermonde matrix by
\begin{equation*}
\Vand(\lambda_1,\dots,\lambda_m;k,K) =
\begin{pmatrix}
\lambda_1^{k-1} & \lambda_1^k  & \cdots & \lambda_{1}^{K-1} \\
\lambda_2^{k-1} & \lambda_2^k  & \cdots & \lambda_{2}^{K-1} \\
\vdots   & \vdots & \ddots  & \vdots \\
\lambda_m^{k-1} & \lambda_m^k  &  \hdots   & \lambda_{m}^{K-1} 
\end{pmatrix}
\in \mathbb{C}^{m \times (K-(k-1))}.
\end{equation*}
with $\Vand(\lambda_1,\dots,\lambda_m) = \Vand(\lambda_1,\dots,\lambda_m;1,m)$ for the usual Vandermonde matrix.
\end{definition}

\begin{theorem}[Reconstruction-property system matrix] \label{2:Theorem_Reconst_System_matrix_A} 
Let the system matrix $A$ be given to the data $x_0,\dots,x_m \in \mathbb{C}^n$.
Then the two following assertions are equivalent:
\begin{enumerate}
\item[(i)] The system matrix $A$ has the following properties:
\begin{itemize}
\item[a)] $\ker(X) \subseteq \ker(Y)$.
\item[b)] $A$ is diagonalizable, where the non-zero eigenvalues are distinct.
\item[c)] $\rank(Y) = r_1$.
\end{itemize} 
\item[(ii)] There are distinct numbers $0 \neq \lambda_1,\dots,\lambda_{r_1} \in \mathbb{C}$, coefficients $0 \neq b_1,\dots,b_{r_2} \in \mathbb{C}$ and linearly independent vectors $v_1,\dots,v_{r_2}$ with $r_1 \leq r_2 \leq n$ and $r_1 \leq m$ such that the following identities hold for $k = 1,\dots,m$:
\begin{equation*} 
x_0 = \sum_{j=1}^{r_2} b_j v_j,
\qquad
x_k = \sum_{j=1}^{r_1} \lambda_j^k b_j v_j.
\end{equation*}
\end{enumerate}
In this case, the non-zero distinct numbers $\lambda_1,\dots,\lambda_{r_1}$ are the eigenvalues of the system matrix and the related (scaled) vectors $v_1,\dots,v_{r_1}$ are the corresponding eigenvectors. The remaining vectors $v_{r_1 + 1},\dots, v_{r_2}$ are eigenvectors of the system matrix to the eigenvalue zero.
\end{theorem}
\begin{proof}
``$(i) \implies (ii)$''.
By condition $c)$, the system matrix has at most $r_1$ non-zero eigenvalues, because $\rank(A) \leq \min\{\rank(X^+),\rank(Y)\} \leq r_1$.
As the system matrix is diagonalizable by assumption $b)$, there exist non-zero distinct eigenvalues $\lambda_1,\dots,\lambda_r$ with $r \leq r_1$ and zero eigenvalues $\lambda_{r+1},\dots,\lambda_n$ with corresponding eigenvectors $v_1,\dots,v_r,v_{r+1},\dots,v_n$.
Rewriting into matrices leads to $W^{-1} A W = \Lambda$ for $W = \begin{bmatrix} v_1 & v_2 & \dots & v_n \end{bmatrix}$ and $\Lambda = \diag(\lambda_1,\dots,\lambda_n)$.
By \cref{2:Corollary_AX=Y}, the first assumption $a)$ is equivalent to $AX = Y$,  which implies the identity
\begin{equation*}
x_k 
= A^k x_0 
= W \Lambda^k W^{-1} W b
= W \Lambda^k b
= \sum_{j=1}^n \lambda_j^k b_j v_j,
\end{equation*}
for $k = 0,1,\dots,m$, where $b=(b_1,\dots,b_n)^T$ contains the coefficients of the linear combination of $x_0$ in the eigenvector basis, i.e., $b = W^{-1} x_0$.
Some of the coefficients $b_j$ may be zero, such that we obtain, after reordering, the identity 
\begin{equation*} 
x_k = \sum_{j=1}^{\tilde{r}} \lambda_j^k b_j v_j,
\end{equation*}
for $k = 1,\dots,m$ and $b_1,\dots,b_{\tilde{r}} \neq 0$ with $\tilde{r} \leq r \leq r_1$.
Since the rank of $Y$ is $r_1$ by condition $c)$, the data points $x_1,\dots,x_m$ span an $r_1$-dimensional vector subspace.
As the vectors $v_1,\dots,v_{\tilde{r}}$ are linearly independent, the sum have to has at least $r_1$ terms. 
Hence, $\tilde{r} = r =  r_1$ and $b_1,\dots,b_{r_1} \neq 0$ as well as $\lambda_1,\dots,\lambda_{r_1} \neq 0$.
Finally, the first data point can be expressed by the remaining non-zero coefficients $b_1,\dots,b_{r_1},b_{r_1+1},\dots,b_{r_2}$ through
\begin{equation*} 
x_0 = \sum_{j=1}^{r_2} b_j v_j.
\end{equation*}

``$(ii) \implies (i)$''. 
First, we define the following matrices
\begin{align*}
&W_{r_1} = \begin{bmatrix} v_1&\dots&v_{r_1} \end{bmatrix} \in \mathbb{C}^{n \times r_1}
&& W_{r_2} = \begin{bmatrix} v_1&\dots&v_{r_2} \end{bmatrix} \in \mathbb{C}^{n \times r_2}
\\
&K_{r_1} = \diag(b_1,\dots,b_{r_1}) \in \mathbb{C}^{r_1 \times r_1}
&& K_{r_2} = \diag(b_1,\dots,b_{r_2}) \in \mathbb{C}^{r_2 \times r_2}
\\
&M_{r_1} = \Vand(\lambda_1,\dots,\lambda_{r_1};1,m)  \in \mathbb{C}^{r_1 \times m}
&& M_{r_2} = \Vand(\lambda_1,\dots,\lambda_{r_1},0,\dots,0;1,m) \in \mathbb{C}^{r_2 \times m}
\\
&\Lambda_{r_1} = \diag(\lambda_1,\dots,\lambda_{r_1}) \in \mathbb{C}^{r_1 \times r_1}
&&\Lambda_{r_1 \times r_2} = \begin{bmatrix} \Lambda_{r_1} & 0 \end{bmatrix} \in \mathbb{C}^{r_1 \times r_2}.
\end{align*}
Now, we can rewrite the data matrices $X$ and $Y$ with the above introduced notation by
\begin{equation*}
X = W_{r_2} K_{r_2} M_{r_2},
\qquad
Y = W_{r_1} K_{r_1} \Lambda_{r_1} M_{r_1} 
= W_{r_1} K_{r_1} \Lambda_{r_1 \times r_2} M_{r_2}. 
\end{equation*}
Consider the linear operator $A_* = W_{r_1} \Lambda_{r_1 \times r_2} W_{r_2} ^+ \in \mathbb{C}^{n \times n}$. This operator exactly connects the data, since
\begin{equation*}
A_* X = W_{r_1} \Lambda_{r_1 \times r_2} W_{r_2} ^+ W_{r_2} K_{r_2} M_{r_2}
= W_{r_1} \Lambda_{r_1 \times r_2} K_{r_2} M_{r_2}
= W_{r_1} K_{r_1} \Lambda_{r_1 \times r_2} M_{r_2} 
= Y.
\end{equation*}
As a result, the system matrix $A = Y X^+$ is exact, too, because it minimizes the problem $\lVert A X - Y \rVert_F$, i.e., $A X = Y$. 
By \cref{2:Corollary_AX=Y}, the first assertion $\ker(X) \subseteq \ker(Y)$ is proven. 

The second claim, $\rank(Y) = r_1$, follows from the application of simple rank inequalities onto $Y = W_{r_1} K_{r_1} \Lambda_{r_1} M_{r_1}$ together with the conditions in $(ii)$.

As a consequence, we have $\rank(A) \leq r_1$, which implies that $A$ has at most $r_1$ non-zero eigenvalues.
Now, if we prove that $\lambda_1,\dots,\lambda_{r_1}$ are these eigenvalues, then the proof is complete, since this implies the diagonalizability of the system matrix with distinct non-zero eigenvalues (\cref{2:Corollary_DiagoA}).
More precisely, we have to show that
\begin{equation*}
A W_{r_1} K_{r_1} = W_{r_1} K_{r_1} \Lambda_{r_1}.
\end{equation*}
Note that the column vectors of $W_{r_1} K_{r_1}$ (representing the eigenvectors) are already non-zero by assumption.
To prove the equality we consider two cases:

1. case: $r_1 < m$. Defining the matrix $\tilde{M}_{r_1} = \Vand(\lambda_1,\dots,\lambda_{r_1};2,m) \in \mathbb{C}^{r_1 \times m-1}$ (i.e., without the first column), then we obtain by the exactness of the system matrix
\begin{equation*}
A W_{r_1} K_{r_1} \tilde{M}_{r_1} = W_{r_1} K_{r_1} \Lambda_{r_1} \tilde{M}_{r_1}.
\end{equation*}
Since the eigenvalues are distinct and $r_1 < m$, the matrix $\tilde{M}_{r_1}$ has a right inverse (which is given by the pseudoinverse), hence, we arrive at the desired assertion.

2. case: $r_1 = m$. From $\rank(Y) = m$ follows $\ker(Y) = \{0\}$ and, consequently, $\ker(X) = \{0\}$ by the already proven condition $a)$. 
Hence, $\rank(X) = m$, too and by \cref{2:Lemma_DimImKerA}, the system matrix has $\rank(A) = m$ and $\dim~\ker(A) = n-m$. 
As the second identity geometrically implies $\im(Y) \subseteq \langle v_1,\dots,v_m \rangle$, we obtain the following relationship
\begin{equation*}
\im(A) \subseteq \im(Y) \subseteq \langle v_1,\dots,v_m \rangle.
\end{equation*}
Since $\rank(A) = m$ and $v_1,\dots,v_m$ are linearly independent, it follows $\im(A) = \langle v_1,\dots,v_m \rangle$.
Consequently, the remaining linearly independent vectors $v_{m+1},\dots,v_{r_2}$ belong to the kernel of the system matrix $A$.
Rewriting the matrix $X$ into $X= W_{r_1} K_{r_1} M_{r_1} + q e_1^T$ with $q = \sum_{j=m+1}^{r_2} b_j v_j$ with the first standard basis vector $e_1$, we obtain
\begin{equation*}
W_{r_1} K_{r_1} \Lambda_{r_1} M_{r_1} = Y = A X = A W_{r_1} K_{r_1} M_{r_1} +  A q e_1^T = A W_{r_1} K_{r_1} M_{r_1},
\end{equation*}
because $q \in \ker(A)$.
Multiplying this equation by $M_{r_1}^+$ from the right side, we obtain the desired algebraic identity, since $M_{r_1}^+$ is a right inverse.
\end{proof}

\section{Companion Dynamic Mode Decomposition (CDMD)}
In 2008, Schmid and Sesterhenn presented the first version of DMD~\cite{schmid:2008:APS}. 
This variant will be referred to as Companion Dynamic Mode Decomposition (CDMD).
CDMD has been investigated by experimental and numerical data by Schmid~\cite{schmid:2010:dmd_numerical_data,schmid:2011:applications_of_dmd}.
The first approaches of a theoretic analysis were performed by Rowley et al.~\cite{rowley:2009:spectral_nonlinear_flow}.
Before we discuss the derivation of CDMD, the companion matrix will be defined, which is eponymous for this variant of DMD.

\begin{definition}
For a vector $c = (c_0,\dots,c_{m-1})^T \in \mathbb{C}^m$, we define the matrix $C_c \in \mathbb{C}^{m \times m}$ of the form
\begin{equation*}
\begin{pmatrix}
0 & 0  & \cdots & c_0 \\
1 & 0  & \cdots & c_1 \\
\vdots   & \ddots & \ddots  & \vdots \\
0  & \hdots   &    1   & c_{m-1} 
\end{pmatrix}
\end{equation*}
as the companion matrix to the vector $c$.
\end{definition}

For the derivation of CDMD, we consider a given data set $x_0,\dots,x_m \in \mathbb{C}^n$. 
Again, we rewrite the data points into the matrices $X = \begin{bmatrix} x_0 & \dots & x_{m-1} \end{bmatrix}$ and $Y = \begin{bmatrix} x_1 & \dots & x_m \end{bmatrix}$.
Now, instead of using the system matrix $A$, which connects the matrices $X$ and $Y$ (from the left), i.e., $AX \approx Y$, another approach is to look for a right-hand multiplied matrix $C \in \mathbb{C}^{m \times m}$ such that $Y \approx X C$.
By construction, we can choose $C$ as a companion matrix $C_c$ and, hence, the problem reduces to find a vector $c$ such that
\begin{equation*}
Y = X C_c + q e_m^{T}
\end{equation*}
with a minimal error $q \in \mathbb{C}^{n}$, where $e_m$ represent the last standard basis vector.
Under the assumption $n \gg m$, the companion matrix $C_c \in \mathbb{C}^{m \times m}$ is substantially lower dimensional than the system matrix $A \in \mathbb{C}^{n \times n}$.
In addition, with decreasing error $q$ the companion matrix approximates the spectral-theoretic properties of $A$. 
In fact, for an eigenvector $v$ of $C_c$ to the eigenvalue $\lambda$ the transformed eigenvector $\vartheta = X v$ satisfies
\begin{equation*}
A \vartheta = A X v \approx X C_c v = X \lambda v = \lambda \vartheta.
\end{equation*}
In sum, the computation of $C_c$ is reduced to the calculation of the associated vector $c = (c_0,\dots,c_{m-1})^T$. 
This vector minimizes the error $q = x_m - Xc$ and hence only need to solve the following minimization problem
\begin{equation*}
\min_{c \in \mathbb{C}^{m}} \lVert X c - x_m \rVert_2^2,
\end{equation*}
which will be solved by $c = X^+ x_m$. 
Consequently, we can formulate CDMD as an algorithm, however, before, we define the companion matrix to a given data set. 

\begin{definition} For data $x_0,\dots,x_m \in \mathbb{C}^n$ with associated matrix $X = \begin{bmatrix} x_0 & \dots & x_{m-1} \end{bmatrix}$, we call the companion matrix $C_c$ to the vector $c = X^+ x_m \in \mathbb{C}^m$ the companion matrix to the data $x_0,\dots,x_m$.
\end{definition}

\begin{algorithm}[Companion Dynamic Mode Decomposition (CDMD)] \label{3:Algorithmus_CDMD} \leavevmode\newline 
Input: Data $x_0,\dotso,x_m \in \mathbb{C}^n$. \\
Output: DMD eigenvalues $\lambda_1,\dotso,\lambda_m \in \mathbb{C}$, DMD modes $\vartheta_1,\dotso,\vartheta_m \in \mathbb{C}^n$, and DMD amplitudes $a_1,\dots,a_m \in \mathbb{C}$.
\begin{enumerate}
\item Define the matrices $X:= \begin{bmatrix} x_0&\dotso&x_{m-1} \end{bmatrix}, Y:= \begin{bmatrix} x_1&\dotso&x_{m} \end{bmatrix} \in \mathbb{C}^{n \times m}$.
\item Compute the $c = X^+x_m \in \mathbb{C}^m$.
\item Construct the companion matrix $C_c \in \mathbb{C}^{m \times m}$ to the vector $c$.
\item Compute the DMD eigenvalues $\lambda_1,\dots,\lambda_m$ and eigenvectors $W^{-1} = \begin{bmatrix} v_1&\dotso&v_m \end{bmatrix}$ of $C_c$ to the vector $c$.
\item Calculate DMD modes $\Theta = \begin{bmatrix} \vartheta_1&\dotso&\vartheta_m \end{bmatrix} = X W^{-1} \in \mathbb{C}^{n \times m}$.
\item Compute the DMD amplitudes by $K = \diag(a_1,\dots,a_m) =(\Vand(\lambda_1,\dots,\lambda_m) W^{-1})^+$.
\end{enumerate} 
\end{algorithm}

The algorithm described here differs from the standard literature as we introduce the necessary concept of DMD amplitudes.
Note, that the matrix $K$ in step $6$, which defines the DMD amplitudes, may not be diagonal. 
However, if the eigenvalues are distinct, then $K$ is a diagonal matrix, as we will observe later.
The current definition of the DMD amplitudes seems to be obscure and not very intuitive. 
Later on, we will show that these DMD amplitudes are the right scaling factors for an exact reconstruction.
In addition, we will prove that this choice equals $(a_1,\dots,a_m)^T = \Theta^+ x_0$ under some further assumption.

For a theoretical investigation, we need a more compact representation of the companion matrix $C_c$ from \cref{3:Algorithmus_CDMD}. 
To this end, consider the following minimization problem
\begin{equation*}
\min_{C \in \mathbb{C}^{m \times m}} \lVert X C - Y \rVert_F^2,
\end{equation*}
which is closely related to the construction of the companion matrix.
An explicit solution of the minimization problem is given by  $C = X^+ Y \in \mathbb{C}^{m \times m}$.
This matrix will be captured in the following definition.

\begin{definition} For data $x_0,\dots,x_m \in \mathbb{C}^n$ with associated matrices $X = \begin{bmatrix} x_0 & \dots & x_{m-1} \end{bmatrix}$ and $Y = \begin{bmatrix} x_1 & \dots & x_m \end{bmatrix}$, we call the matrix $C = X^+ Y \in \mathbb{C}^{n \times n}$ the twisted system matrix to the data $x_0,\dots,x_m$.
\end{definition}

\begin{remark}
At first glance, the two approaches seems to be equivalent. 
However, note that the twisted system matrix $C$ and the companion matrix $C_c$ are not necessarily equal or similar. This fact can be observed by a rank truncation of the matrices $X$ and $Y$.
In particular, the companion matrix $C_c$ has at least rank $m-1$, because the first $m-1$ columns are linearly independent. 
However, the rank of the matrix $C = X^+ Y$ depends only on the matrices $X$ and $Y$ and therefore it may be less than $m-1$, which implies a non-similarity in general.
\end{remark}

The following lemma and corollary characterize the relations between the three objects $A, C$, and $C_c$.

\begin{lemma} \label{3:Lemma_C_C_c_A} 
Let the system matrix $A$, the twisted system matrix $C$, and the companion matrix $C_c$ be given to the data $x_0,\dots,x_m \in \mathbb{C}^n$ as well as the projections $P_X = X X^+$ and $P_{X^*} = X^+ X$ onto $\im(X) \subseteq \mathbb{C}^n$ and $\im(X^*) \subseteq \mathbb{C}^m$, respectively. Then the following assertions hold:
\begin{enumerate}
\item[(1)] $X^+ q = 0$, especially $Aq = 0$.
\item[(2)] $P_{X^*} C_c P_{X^*} = C P_{X^*} = X^+ A X$.
\item[(3)] $X C_c X^+ = X C X^+ = P_{X} A$.
\end{enumerate}
\end{lemma}
\begin{proof} The identities are proven by simple algebraic manipulations.
\end{proof}

A direct consequence of this lemma is that the twisted system matrix and the companion matrix are equal, if the first $m$ snapshots are linearly independent. 
In addition, in the case of $n \gg m$, it illustrates that the companion matrix $C_c \in \mathbb{C}^{m \times m}$ is  a low-dimensional representation of the system matrix $A \in \mathbb{C}^{n \times n}$.

\begin{corollary} \label{3:Corollary_C_C_c_A}
Let the system matrix $A$, the twisted system matrix $C$ and the companion matrix $C_c$ be given to the data $x_0,\dots,x_m \in \mathbb{C}^n$. If $x_0,\dots,x_{m-1}$ are linearly independent, then the following assertions hold:
\begin{enumerate}
\item[(1)] $C_c = C$.
\item[(2)] $C_c = X^+ A X$.
\end{enumerate}
\end{corollary}
\begin{proof} Since $x_0,\dots,x_{m-1}$ are linearly independent, the identity $X^+ X = I$ holds, which implies the assertions by \cref{3:Lemma_C_C_c_A}
\end{proof}

The property of a low-dimensional representation suggests that CDMD inherits the characteristic reconstruction of the system matrix.
Indeed, the following theorem proves this fact using the new concept of DMD amplitudes from \cref{3:Algorithmus_CDMD}. 
In the literature a similar theorem is known~\cite[Theorem 1]{rowley:2009:spectral_nonlinear_flow} that does not account for amplitudes. 
Hence, the choice of the corresponding modes is not appropriate for the reconstruction as these are computed up to a scaling (since they stem from eigenvectors).  

\begin{theorem}[Reconstruction-property of CDMD] \label{3:Thoerem_Reconst_CDMD}
Let DMD eigenvalues $\lambda_1,\dots,\lambda_m$, DMD amplitudes $a_1,\dots,a_m$, and DMD modes $\vartheta_1,\dots,\vartheta_m$ be given by \cref{3:Algorithmus_CDMD} to the data $x_0,\dots,x_m \in \mathbb{C}^n$. If the DMD eigenvalues $\lambda_1,\dots,\lambda_m$ are distinct, then the following identities hold:
\begin{equation*}
x_k = \sum_{j=0}^m \lambda_j^k a_j \vartheta_j,
\qquad
x_m = \sum_{j=0}^m \lambda_j^m a_j \vartheta_j + q,
\end{equation*}
for $k = 0,\dotso,m-1$ and $q = x_m - Xc$.
\end{theorem}
\begin{proof}
Since the eigenvalues are distinct, the companion matrix $C_c$ will be diagonalized by the Vandermonde matrix:
\begin{equation*}
\Vand(\lambda_1,\dots,\lambda_m) C_c \Vand(\lambda_1,\dots,\lambda_m)^{-1} = \Lambda.
\end{equation*}
However, the eigenvectors $\begin{bmatrix} z_1 & \dots & z_m \end{bmatrix} = \Vand(\lambda_1,\dots,\lambda_m)^{-1}$ may not coincide with the eigenvectors $v_1,\dots,v_m$ of the companion matrix produced by \cref{3:Algorithmus_CDMD}. 
As the eigenvalues are distinct, there exist scaling factors $\alpha_1,\dots,\alpha_m \in \mathbb{C}$ such that for $j = 1,\dots,m$:
\begin{equation*}
\alpha_j v_j = z_j.
\end{equation*}
For the eigenvectors $v_1,\dots,v_m$ and scaling factors $\alpha_1,\dots,\alpha_m$ we define the matrices $W^{-1} = \begin{bmatrix} v_1 & \dots & v_m \end{bmatrix}$ and $K = \diag(\alpha_1,\dots,\alpha_m$), respectively.
As a result, we get the following connection
\begin{equation*}
W^{-1} K = \Vand(\lambda_1,\dots,\lambda_m)^{-1}.
\end{equation*}
Consequently, $K = (\Vand(\lambda_1,\dots,\lambda_m) W^{-1})^{-1}$, which shows that $K$ equals the matrix in step $6$ of \cref{3:Algorithmus_CDMD} consisting of the DMD amplitudes, i.e., $\alpha_j = a_j$ for $j =1,\dots,m$.
Multiplying the above equation by the matrix $X$, we obtain
\begin{equation*}
X W^{-1} K = X \Vand(\lambda_1,\dots,\lambda_m)^{-1}.
\end{equation*}

By this notation the DMD modes $\Theta = \begin{bmatrix} \vartheta_1&\dots&\vartheta_m 
\end{bmatrix}$ are given by $\Theta = X W^{-1}$ and, therefore, the first assertions follows by rearranging the above equation to $X = \Theta K \Vand(\lambda_1,\dots,\lambda_m)$.
The second identity is a result of the following calculation
\begin{align*}
x_m 
= Xc + q 
= XCe_m^T + q 
&= \Theta K \Vand(\lambda_1,\dots,\lambda_m) C e_m^T + q \\
&= \Theta K \Lambda \Vand(\lambda_1,\dots,\lambda_m) e_m^T + q \\
&= \sum_{j=0}^m \lambda_j^m a_j \vartheta_j + q.
\end{align*}
\end{proof}

\begin{corollary} \label{3:Corollary_Ampltidues_CDMD_Alternative} 
Let DMD eigenvalues $\lambda_1,\dots,\lambda_m$, DMD amplitudes $a_1,\dots,a_m$, and DMD modes $\vartheta_1,\dots,\vartheta_m$ be given by \cref{3:Algorithmus_CDMD} to the data $x_0,\dots,x_m \in \mathbb{C}^n$. If the DMD eigenvalues $\lambda_1,\dots,\lambda_m$ are distinct and $x_0,\dots,x_{m-1}$ are linearly independent, then the DMD amplitudes $a = (a_1,\dots,a_m)^T$ can be calculated by \begin{equation*}
a = \Theta^+ x_0.
\end{equation*} 
\end{corollary}
\begin{proof}
By \cref{3:Thoerem_Reconst_CDMD}, we obtain the relation $x_0 = \Theta a$,
where $\Theta = \begin{bmatrix} \vartheta_1&\dots&\vartheta_m \end{bmatrix} = X W^{-1}$ with $W^{-1} = \begin{bmatrix} v_1 & \dots & v_m \end{bmatrix}$. Since the matrix $X$ and $W^{-1}$ have full rank, the DMD modes will be linear independent and, consequently, there exist a left-inverse of $\Theta$, which is given by its pseudoinverse $\Theta^+$. Hence, $ a = \Theta^+ x_0$.
\end{proof}

Even though, the method of CDMD is mathematically correct, a practical implementation leads to an ill-conditioned algorithm~\cite{schmid:2010:dmd_numerical_data}.
The reason for this is the external computation of the vector $c$ (which define companion matrix $C_c$) that leads to unsatisfied approximation properties of the system matrix $A$.
This problem can be tackled by using the robust singular value decomposition (SVD), and will be discussed in the next section.

\section{SVD-Dynamic Mode Decomposition (SDMD)}
In 2010, the algorithm of DMD was modified radically by Schmid. 
He published a variant of DMD based on a reduced singular value decomposition~\cite{schmid:2010:dmd_numerical_data}. 
For this reason, we will denote this method by SVD-Dynamic Mode Decomposition (SDMD). 
The result was a robust and stable algorithm, which serves as a basis for the most modern version of DMD. 
We start this section by deducing this algorithm.

Let the system matrix $A$ to the data $x_0,\dots,x_m \in \mathbb{C}^n$ be given as well as the reduced singular value decomposition of $X$ with $r= \rank(X)$ and $X=U\Sigma V^* \in \mathbb{C}^{n \times m}$,
where $U\in \mathbb{C}^{n \times r}$, $V \in \mathbb{C}^{m \times r}$ and $\Sigma \in \mathbb{R}^{r \times r}$. 
Utilizing the transformation matrix $U$, we construct the low-dimensional representation $S$ of the system matrix $A$ by
\begin{equation*}
S := U^* A U \in \mathbb{C}^{r \times r}.
\end{equation*}
For an explicit calculation of $S$, it is necessary to avoid the computation of the system matrix $A$.
To this end, we calculate
\begin{equation*}
S = U^* A U 
= U^* A U \Sigma V^* V \Sigma^{-1}
= U^* A X V \Sigma^{-1}
= U^* Y V\Sigma^{-1}.
\end{equation*}
Now, we compute the eigenvalues $\lambda_i$ and eigenvectors $v_i$ of the matrix $S$ and finally transform them by the matrix $U$ into
\begin{equation*}
\vartheta_i := Uv_i \in \mathbb{C}^n,
\end{equation*}
in order to obtain an approximation of the eigenvalues and eigenvectors of the system matrix:
\begin{equation*}
A \vartheta_i = A U v_i \approx U U^* A U v_i = U S v_i = U \lambda v_i = \lambda \vartheta_i. 
\end{equation*}
Consequently, the algorithm can be formulated as follows.

\begin{algorithm}(SVD-Dynamic Mode Decomposition) \leavevmode\newline 
Input: Data $x_0,\dotso,x_m \in \mathbb{C}^n$.
\label{4:Algorithmus_SDMD}  \\ 
Output: DMD eigenvalues $\lambda_1,\dotso,\lambda_r \in \mathbb{C}$, DMD modes $\vartheta_1,\dotso,\vartheta_r \in \mathbb{C}^n$, and DMD amplitudes $a_i,\dots,a_r \in \mathbb{C}$.
\begin{enumerate}
\item Define the matrices $X:=(x_0,\dotso,x_{m-1}), Y:=(x_1,\dotso,x_m) \in \mathbb{C}^{n \times m}$.
\item Compute the reduced singular value decomposition $X=U\Sigma V^*$ with $r = \rank(X)$.
\item Define the DMD matrix $S:=U^*YV\Sigma^{-1} \in \mathbb{C}^{r \times r}$.
\item Compute the DMD eigenvalues $\lambda_i$ and eigenvectors $v_i$ of $S$.
\item Calculate the DMD modes $\vartheta_i=Uv_i \in \mathbb{C}^n$ and define $\Theta = (\vartheta_1,\dots,\vartheta_r) \in \mathbb{C}^{n \times r}$.
\item Compute the DMD amplitudes $a=\Theta^+ x_0 \in \mathbb{C}^r$ with $a=(a_1,\dots,a_r)^T$.
\end{enumerate} 
\end{algorithm}

In \cref{4:Algorithmus_SDMD}, the DMD amplitudes will be defined intuitively by a best-fit linear combination of the first snapshots $x_0$ in the modes selection (compare \cref{3:Corollary_Ampltidues_CDMD_Alternative}).
For a deeper understanding of SDMD, we first characterize the connection to CDMD, which will be examined in the following lemma.

\begin{lemma} \label{4:Lemma_CS_similar}
Let the companion matrix $C_c$ and the DMD matrix $S$ be given by \cref{4:Algorithmus_SDMD} to the data $x_0,\dotso,x_m \in \mathbb{C}^n$. Then the following identity holds:
\begin{equation*}
    S = (\Sigma V^*) C_c (V \Sigma^{-1}).
\end{equation*}
In particular, if $x_0,\dotso,x_{m-1}$ are linearly independent, then the matrices $C_c$ and $S$ are similar.
\end{lemma}
\begin{proof}
By \cref{3:Lemma_C_C_c_A} and the reduced singular value decomposition $X = U \Sigma V^*$, we obtain
\begin{align*}
S
= U^* Y V \Sigma^{-1} 
&= U^* (X C_c + q e_m^T) V \Sigma^{-1} \\
&= \Sigma V^* X^+ ( X C_c + q e_m^T ) V \Sigma^{-1}
= \Sigma V^* X^+ X C_c V \Sigma^{-1} 
= \Sigma V^* C_c V \Sigma^{-1}
\end{align*}
\end{proof}

\begin{corollary}
Let the data $x_0,\dotso,x_m \in \mathbb{C}^n$ be given. If $x_0,\dotso,x_{m-1}$ are linearly independent, then CDMD and SDMD produce the same DMD eigenvalues.
\end{corollary}

The corollary suggests that the reconstruction property of CDMD from \cref{3:Thoerem_Reconst_CDMD} will be also transferred onto SDMD.
In the case of distinct DMD eigenvalues, the DMD modes associated to equal eigenvalues only differ by a scaling factor.
Therefore, the DMD modes of SDMD have to be rescaled for an exact reconstruction. 
The following theorem shows the reconstruction property of SDMD, where the scaling factors are given by the DMD amplitudes.

\begin{theorem}[Reconstruction-property SDMD] \label{4:Thoerem_Reconst_SDMD}
Let DMD eigenvalues $\lambda_1,\dots,\lambda_m$, \allowbreak DMD amplitudes $a_1,\dots,a_m$, and DMD modes $\vartheta_1,\dots,\vartheta_m$ be given by \cref{4:Algorithmus_SDMD} to the data $x_0,\dotso,x_m \in \mathbb{C}^n$, where $x_0,\dotso,x_{m-1}$ are linearly independent. If the DMD eigenvalues are distinct, then the following identities hold:
\begin{equation*}
x_k = \sum_{j=1}^m {\lambda_j^k a_j \vartheta_j}, 
\hspace{20pt}
x_m = \sum_{j=1}^m {\lambda_j^m a_j \vartheta_j} + q
\end{equation*}
for $k = 0,\dotso,m-1$ and $q = x_m - Xc$ with $c = X^+x_m$.
\end{theorem}
\begin{proof} By the notation of the proof of \cref{3:Thoerem_Reconst_CDMD} and \cref{4:Lemma_CS_similar} we obtain
\begin{align*}
W S W^{-1} 
= \Lambda 
&= \Vand(\lambda_1,\dots,\lambda_m) C_c \Vand(\lambda_1,\dots,\lambda_m)^{-1} \\
&= \Vand(\lambda_1,\dots,\lambda_m) (\Sigma V^*)^{-1} S (\Sigma V^*) \Vand(\lambda_1,\dots,\lambda_m)^{-1}, 
\end{align*}
where the matrix $W^{-1}$ contains the eigenvectors of $S$. 
Consider a scaling matrix $K=\text{diag}(\alpha_1,\dotso,\alpha_m)$ that satisfies the equation
\begin{equation*}
W^{-1} K = (\Sigma V^*) \Vand(\lambda_1,\dots,\lambda_m)^{-1}.
\end{equation*}
Since the scaling matrix $K$ adjusts eigenvectors (in the same one-dimensional eigenspace), the scaling factors are $\alpha_1,\dots,\alpha_m \neq 0$ and, therefore, the matrix $K$ is invertible. By multiplying the above equation with $U$ from the left, we obtain
\begin{equation*}
\Theta K = U W^{-1} K = U (\Sigma V^*) \Vand(\lambda_1,\dots,\lambda_m)^{-1} = X \Vand(\lambda_1,\dots,\lambda_m)^{-1},
\end{equation*}
where $\Theta = U W^{-1}$ are the DMD modes (by \cref{4:Algorithmus_SDMD}). 
In sum, we obtain
\begin{equation*}
X = \Theta K \Vand(\lambda_1,\dots,\lambda_m),
\end{equation*}
which shows the first identity concerning the scalings $\alpha_1,\dots,\alpha_m$. The second statement follows analogously by
\begin{align*}
x_m 
= X C e_m^T + q 
&= \Theta K \Vand(\lambda_1,\dots,\lambda_m) C e_m^T + q \\
&= \Theta K \Lambda \Vand(\lambda_1,\dots,\lambda_m) e_m^T + q \\
&=  \sum_{j=1}^m {\lambda_j^m a_j \vartheta_j} + q.
\end{align*}
It misses to show that the scaling factors $\alpha_i$ are given by the DMD amplitudes $a_i$. 
Rewriting $\alpha = (\alpha_1,\dots,\alpha_m)^T$ and using the identity $X \Vand(\lambda_1,\dots,\lambda_m)^{-1} e = x_0$ (from \cref{3:Thoerem_Reconst_CDMD}), we obtain: 
\begin{align*}
\alpha 
= K e 
&= W \Sigma V^* \Vand(\lambda_1,\dots,\lambda_m)^{-1} e \\
&= W U^* X \Vand(\lambda_1,\dots,\lambda_m)^{-1} e
= W U^* x_0
= \Theta^{-1} x_0
= a.
\end{align*}
\end{proof}

\begin{remark} \label{4:Remark_EigenvalueEquation_with_Projection}
SDMD is characterized by the robust singular value decomposition and the reconstruction property. 
However, the spectral-theoretical connection of SDMD to the system matrix is not clear. 
Therefore, let the system matrix $A$ as well as DMD eigenvalues $\lambda_1,\dots,\lambda_m$ and DMD modes $\vartheta_1,\dots,\vartheta_m$ be given by \cref{4:Algorithmus_SDMD} to the data $x_0,\dotso,x_m \in \mathbb{C}^n$. Then the following equation hold:
\begin{equation*}
\lambda_i \vartheta_i
= \lambda_i U v_i
= U S v_i
= U U^* Y V \Sigma^{-1} v_i
= U U^* Y V \Sigma^{-1} U^* U v_i
= U U^* Y X^+ \vartheta_i
= U U^* A \vartheta_i.
\end{equation*}
We observe that the eigenvalue equation is correct up to the projection $P_X = X X^+ = U U^*$ onto the image of $X$. 
\end{remark}

The following proposition presents characterizations for the equality of the eigenvector equation of \cref{4:Remark_EigenvalueEquation_with_Projection}.

\begin{proposition} \label{4:Proposition_SDMD_A_Characterization}
Let the system matrix $A$ as well as DMD eigenvalues $\lambda_1,\dots,\lambda_m$ and DMD modes $\vartheta_1,\dots,\vartheta_m$ be given by \cref{4:Algorithmus_SDMD} to the data $x_0,\dotso,x_m \in \mathbb{C}^n$. In addition, let the error $q = x_m - Xc$ with  $c = X^+ x_m$ be given. Then the following assertions are equivalent:
\begin{itemize}
\itemsep0pt
\item[(i)] The DMD mode $\vartheta_i$ is an eigenvector of the system matrix $A$ to the eigenvalue $\lambda_i$.
\item[(ii)] $A \vartheta_i \in \langle x_0,\dots,x_{m-1} \rangle$.
\item[(iii)] $ \langle e_m, X^+ \vartheta_i \rangle \cdot q = 0$.
\item[(iv)] $x_m \in \langle x_0,\dots,x_{m-1}\rangle$ or $\langle e_m,V \Sigma^{-1} v_i \rangle = 0$.
\end{itemize}
\end{proposition}
\begin{proof} Let the companion matrix $C_c$ to the data $x_0,\dots,x_m$ be given.

$(iii) \implies (ii)$ By \cref{3:Lemma_C_C_c_A} we obtain the equality
\begin{equation*}
P_X A \vartheta_i 
= X C_c X^+ \vartheta_i
= X C_c X^+ \vartheta_i + q e_m^T X^+ \vartheta_i
= [XC_c + q e_m^T] X^+ \vartheta_i
= Y X^+ \vartheta_i
= A \vartheta_i,
\end{equation*}
which implies $A \vartheta_i \in \text{im}(X)$, i.e. $A \vartheta_i \in \langle x_0,\dots,x_{m-1} \rangle$.

``$(ii) \implies (i)$'' Let $A \vartheta_i \in \langle x_0,\dots,x_{m-1} \rangle$. Then the (algebraic) eigenvalue equation is trivially satisfied, since the projection $P_X$ can be ignored in the equation of \cref{4:Remark_EigenvalueEquation_with_Projection}.
Furthermore, since the columns of $U$ are orthogonal and $v_i \neq 0$, the vector $\vartheta_i = U v_i$ is non-zero and, therefore, an eigenvector of $A$.

``$(i) \implies (iii)$'' By the preliminary \cref{4:Remark_EigenvalueEquation_with_Projection}, \cref{3:Lemma_C_C_c_A}, and assumption (i), we obtain
\begin{equation*}
0 
=A \vartheta_i - \lambda_i \vartheta_i
= Y X^+ \vartheta_i - P_X A \vartheta_i
= [XC_c + q e_m^T] X^+ \vartheta_i - X C_c X^+ \vartheta_i
= q e_m^T X^+ \vartheta_i.
\end{equation*}

``$(iii) \iff (iv)$'' We reformulate the conditions by
\begin{align*}
\langle e_m, X^+ \vartheta_i \rangle \cdot q = 0
& \iff q = 0 \text{ or } \langle e_m,X^+ \vartheta_i \rangle = 0 \\
& \iff q = 0 \text{ or } \langle e_m , V \Sigma^{-1} U^* U v_i \rangle = 0 \\
& \iff x_m \in \langle x_0,\dots,x_{m-1} \rangle  \text{ or } \langle e_m , V \Sigma^{-1} v_i \rangle = 0.
\end{align*}
\end{proof}

A trivial consequence of the above proposition is the following corollary, which was first proven by Tu et al.~\cite{tu:2014:on_dmd_theorey_and_app}.

\begin{corollary} \label{4:Corollary_SDMD_A_Characterization}
Let the system matrix $A$ as well as DMD eigenvalues $\lambda_1,\dots,\lambda_m$ and DMD modes $\vartheta_1,\dots,\vartheta_m$ be given by \cref{4:Algorithmus_SDMD} to the data $x_0,\dotso,x_m \in \mathbb{C}^n$. If $x_m \in \allowbreak  \langle \allowbreak x_0, \allowbreak \dots,x_{m-1} \allowbreak \rangle $, then the DMD modes and DMD eigenvalues are eigenvectors and eigenvalues of the system matrix $A$, respectively.
\end{corollary}

The previous statement implies $\sigma(S) \subseteq \sigma(A)$. However, the other inclusion is also true for all non-zero eigenvalues of the system matrix $A$~\cite{tu:2014:on_dmd_theorey_and_app}.

\begin{proposition} \label{4:Propostion_AllEigenvaluesCalculated}
Let the system matrix $A$ as well as DMD eigenvalues $\lambda_1,\dots,\lambda_m$ and DMD modes $\vartheta_1,\dots,\vartheta_m$ be given by \cref{4:Algorithmus_SDMD} to the data $x_0,\dotso,x_m \in \mathbb{C}^n$. Then all non-zero eigenvalues $\lambda \neq 0$ of $A$ will be calculated by DMD, i.e., it holds
\begin{equation*}
\sigma(A) \setminus \{0\} \subseteq \sigma(S).
\end{equation*}
\end{proposition}
\begin{proof}
For an arbitrary eigenvector $z$ of $A$ to the eigenvalue $\lambda \neq 0$, we obtain by defining the vector $v:= U^*z$:
\begin{equation*}
S v = U^* Y V \Sigma^{-1} U^* z = U^* Y X^+ z = U^* A_\varphi z = U^* \lambda z = \lambda v.
\end{equation*}
Assume $v = 0$. Then $U^*z=0$ and we obtain
\begin{equation*}
0 
= Y V \Sigma^{-1} U^* z 
= Y X^+ z
= A z
= \lambda z.
\end{equation*}
As $z \neq 0$, it follows $\lambda = 0$, which contradicts the assumptions.
\end{proof}

Consequently, under appropriate assumptions of the data, we obtain the spectral-theoretic relation
\begin{equation*}
\sigma(A) \setminus \{0\} \subseteq \sigma(S) \subseteq \sigma(A).
\end{equation*}
However, the condition $x_m \in \langle x_0,\dots,x_{m-1}\rangle$ in \cref{4:Corollary_SDMD_A_Characterization} is actually never satisfied in the case of $n \gg m$ and, hence, not practically applicable. 

This raises the question, whether the DMD modes can be modified such that we obtain eigenvectors of the system matrix without any assumptions.
A solution to this problem is presented in the next theorem, which is inspired by the assertion $(iii)$ in \cref{4:Proposition_SDMD_A_Characterization}.

\begin{theorem} \label{4:Theorem_Modification_Modes}
Let the system matrix $A$ as well as DMD eigenvalues $\lambda_1,\dots,\lambda_m$ and DMD modes $\vartheta_1,\dots,\vartheta_m$ be given by \cref{4:Algorithmus_SDMD} to the data $x_0,\dotso,x_m \in \mathbb{C}^n$. In addition, let the error $q = x_m - Xc$ with  $c = X^+ x_m$ be given. For $\lambda_i \neq 0$
\begin{equation*}
\zeta_i = \vartheta_i + \frac{1}{\lambda_i} \langle e_m, X^+ \vartheta_i \rangle \cdot q
\end{equation*} 
is an eigenvector of the system matrix $A$ to the eigenvalue $\lambda_i$.
\end{theorem}
\begin{proof} By \cref{3:Lemma_C_C_c_A} and \cref{4:Remark_EigenvalueEquation_with_Projection}, we obtain the eigenvalue equation by
\begin{align*}
A \zeta_i 
&= A \vartheta_i + \frac{1}{\lambda_i} \langle e_m, X^+ \vartheta_i \rangle \cdot A q 
= A \vartheta_i
= Y X^+ \vartheta_i 
= [XC_c + q e_m^T] X^+ \vartheta_i \\
&= XC_cX^+ \vartheta_i + q e_m^T X^+ \vartheta_i 
= P_X A \vartheta_i + \langle e_m, X^+ \vartheta_i \rangle \cdot q 
= \lambda_i \vartheta_i + \langle e_m, X^+ \vartheta_i \rangle \cdot q 
= \lambda_i \zeta_i.
\end{align*}
Assume $\zeta = 0$. Then we obtain
\begin{equation*}
0 = X^+ \zeta = X^+ (\vartheta_i + \frac{1}{\lambda_i} \langle e_m, X^+ \vartheta_i \rangle \cdot q) = X^+ \vartheta_i,
\end{equation*}
and therefore
\begin{equation*}
0 
= U^* Y X^+ \vartheta_i 
= U^* Y V \Sigma^{-1} U^* U v_i
= S v_i
= \lambda v_i.
\end{equation*}
Since $v_i \neq 0$, it follows $\lambda = 0$, which contradicts the assumption.
\end{proof}

The previous proposition characterizes exactly the spectral-theoretic relation between the DMD matrix $S$ and the system matrix $A$. 
In fact, it holds $\sigma(A) \setminus \{0\} \subseteq \sigma(S) \subseteq \sigma(A)$ (without any assumption) and thus
\begin{equation*}
\sigma(A) \setminus \{0\} = \sigma(S) \setminus \{0\}.
\end{equation*}
Consequently, the dynamic behavior of the system matrix $A$ will be completely captured by the low-dimensional DMD matrix $S$. 
However, the DMD modes have to be modified according to \cref{4:Theorem_Modification_Modes} in order to get eigenvectors of the system matrix.
This motivates the formulation of a new variant of DMD with modified DMD modes and possibly new DMD amplitudes such that the reconstruction property is preserved. 
In particular, as \cref{4:Theorem_Modification_Modes} states, only the non-zero eigenvalues can be used.
Nevertheless, these eigenvalues are sufficient to capture the temporal evolution and consequently no dynamical information is lost.

\begin{remark} \label{4:Remark_Modification_Modes_Adjustment}
For the formulation of an efficient algorithm, we need a more compact representation of the modified DMD modes $\zeta_i$, which arises directly from transformations of the eigenvectors of $S$. By \cref{4:Remark_EigenvalueEquation_with_Projection} and \cref{3:Lemma_C_C_c_A}, we rearrange the modified DMD modes by
\begin{align*}
\zeta_i 
&= \vartheta_i + \frac{1}{\lambda_i} \langle e_m, X^+ \vartheta_i \rangle \cdot q 
= \frac{1}{\lambda_i} \lambda_i \vartheta_i + \frac{1}{\lambda_i} q e_m^T X^+ \vartheta_i 
=  \frac{1}{\lambda_i} P_X A \vartheta_i + \frac{1}{\lambda_i} q e_m^T X^+ \vartheta_i \\
&=  \frac{1}{\lambda_i} X C_c X^+ \vartheta_i + \frac{1}{\lambda_i} q e_m^T X^+ \vartheta_i 
=  \frac{1}{\lambda_i} (X C_c + q e_m^T) X^+ \vartheta_i 
= \frac{1}{\lambda_i} Y X^+ \vartheta_i 
= \frac{1}{\lambda_i} Y V \Sigma^{-1} v_i.
\end{align*}
\end{remark}

Defining the DMD modes in this way, we obtain the most modern version of DMD, called Exact-Dynamic Mode Decomposition (EXDMD).
In the following section, we will introduce this variant of DMD.

\section{Exact-Dynamic Mode Decomposition (EXDMD)}
In 2014, Tu et. al.~\cite{tu:2014:on_dmd_theorey_and_app} presented the most modern version of DMD, called Exact-Dynamic Mode Decomposition (EXDMD).
However, the algorithm presented here differs from the standard literature as we use a different definition of the DMD amplitudes.
In addition, we introduce a novel relevant variable: The error scaling.

\begin{algorithm}[Exact-Dynamic Mode Decomposition] \leavevmode\newline 
Input: Data $x_0,\dotso,x_m \in \mathbb{C}^n$.
\label{5:Algorithmus_EXDMD} \\
Output: DMD eigenvalues $\lambda_1,\dotso,\lambda_{r_0} \in \mathbb{C}$, DMD modes $\vartheta_1,\dotso,\vartheta_{r_0} \in \mathbb{C}^n$, DMD amplitudes $a_1,\dots,a_{r_0} \in \mathbb{C}$, and the error scaling $a_0$.
\begin{enumerate}
\item Define the matrices $X:=(x_0,\dotso,x_{m-1}), Y:=(x_1,\dotso,x_m) \in \mathbb{C}^{n \times m}$.
\item Compute the reduced singular value decomposition $X=U\Sigma V^*$ with $r = \rank(X)$.
\item Define the DMD matrix $S:=U^*YV\Sigma^{-1} \in \mathbb{C}^{r \times r}$.
\item Compute the DMD eigenvalues $\lambda_i$ and eigenvectors $v_i$ of $S$.
\item Calculate the DMD modes $\vartheta_i= \frac{1}{\lambda_i} Y V \Sigma^{-1} v_i \in \mathbb{C}^n$ and define $\Theta = (\vartheta_1,\dots,\vartheta_{r_0}) \mathbb{C}^{n \times {r_0}}$.
\item Compute the DMD amplitudes  $a=\Lambda^{-1} \Theta^+ x_1 \in \mathbb{C}^r$ with $a=(a_1,\dots,a_{r_0})^T$.
\item Calculate the error scaling $a_0 = -\sum_{j=1}^{r_0} \frac{1}{\lambda_j} \prod_{\substack{k=1 \\ k \neq i}}^{r_0} \frac{1}{\lambda_j - \lambda_k} \in \mathbb{C}$, if it exists. 
\end{enumerate} 
\end{algorithm}

The introduction of the error scaling $a_0$ and the fundamental change of the definition of the DMD amplitude will be justified by the subsequent theorem, which proves the reconstruction property of EXDMD. 
The reconstruction property of EXDMD, however, differs from the previous ones of CDMD (see \cref{3:Thoerem_Reconst_CDMD}) and SDMD (see \cref{4:Thoerem_Reconst_SDMD}). 
The reason for this is the spectral-theoretic relation of EXDMD to the system matrix, as we will examine in the proof.

\begin{lemma} \label{5:Lemma_kerXkerY_kerXneq0}
Let data $x_0,\dotso,x_m \in \mathbb{C}^n$ with associated matrices $X = \begin{bmatrix} x_0 & \dots & x_{m-1} \end{bmatrix}$ and $Y = \begin{bmatrix} x_1 & \dots & x_{m} \end{bmatrix}$ be given. 
If $ker(X) \subseteq ker(Y)$ and $\ker(X) \neq \{0\}$, then $x_m \in \im(X)$.
In particular, it holds $\im(X) \subseteq \im(Y)$.
\begin{proof}
Since the kernel of $X$ is non-trivial, there exist a vector $v=(v_1,\dotso,v_m)^T \neq 0 \in \mathbb{C}^m$ with $Xv = 0$. Since $\ker(X) \subseteq \ker(Y)$, we obtain $Yv = 0$.

1. case: $v_m \neq 0$. Then the equation $v_1  x_1 + \dots v_m  x_m = 0$ implies
\begin{equation*}
x_m = -\frac{1}{v_m}( v_1 x_1 + \dotso +  v_{m-1} x_{m-1}) 
= -\frac{v_1}{v_m} x_1 - \dotso -\frac{v_{m-1}}{v_m} x_{m-1},
\end{equation*}
and hence the assertion is proven.

2. case: $v_m = 0$. Then we obtain
\begin{equation*}
0 = Y \begin{pmatrix} v_1 \\ \vdots \\ v_{m-1} \\ 0 \end{pmatrix} 
= X \begin{pmatrix} 0 \\ v_1 \\ \vdots \\ v_{m-1} \end{pmatrix} .
\end{equation*}
By assumption $\ker(X) \subseteq \ker(Y)$, we obtain
\begin{equation*}
X \begin{pmatrix} 0 \\ v_1 \\ \vdots \\ v_{m-1} \end{pmatrix} = 0
\implies
Y \begin{pmatrix} 0 \\ v_1 \\ \vdots \\ v_{m-1} \end{pmatrix} = 0.
\end{equation*}
If $v_{m-1} \neq 0$, then we get a representation of $x_m$ analogous to the firstcase.
Otherwise, we repeat this steps as long as an entry $v_{j} \neq 0$ appears.
As $v$ is non-zero, there exists such an entry $v_{j_0} \neq 0$.
\end{proof}
\end{lemma}

\begin{theorem}[Reconstruction property EXDMD] \label{5:Theorem_Reconst-EXDMD}
Let DMD eigenvalues $\lambda_1,\dots,\lambda_r$, DMD amplitudes $a_1,\dots,a_r$, DMD modes $\vartheta_1,\dots,\vartheta_r$, and the error scaling $a_0$ be given by \cref{5:Algorithmus_EXDMD} to the data $x_0,\dotso,x_m \in \mathbb{C}^n$.
If $\ker(X) \subseteq \ker(Y)$ and the DMD eigenvalues are distinct, then the following identities hold:
\begin{equation*}
x_0 = \sum_{j=1}^r {a_j \vartheta_j} + q_0,
\qquad
x_k = \sum_{j=1}^r {\lambda_j^k a_j \vartheta_j},
\end{equation*}
for $k = 1,\dotso,m$ and a vector $q_0$ with
\begin{equation*}
\begin{cases}
q_0 = a_0 \cdot q & r = m \\
q_0 = 0 & r \neq m
\end{cases}
\end{equation*}
where $q = x_m - Xc$ and $c = X^+ x_m$.
\end{theorem}
\begin{proof} By \cref{4:Theorem_Modification_Modes} and \cref{4:Remark_Modification_Modes_Adjustment}, we obtain that the DMD eigenvalues $\lambda_i$ and modes $\vartheta_i$ are eigenvalues and eigenvectors of the system matrix $A$ for $i = 1,\dots,m$.
Since $\ker(X) \subseteq \ker(Y)$, we obtain the relation
\begin{equation*}
\rank(Y) = m - \ker(Y) \leq m - \ker(X) = \rank(X) = r. 
\end{equation*}
In addition, the condition $\im(A) \subseteq \im(Y)$ implies $\rank(A) \leq \rank(Y)$ and consequently $r \leq \rank(Y)$, because $A$ has $r$ non-zeros distinct eigenvalues.
Hence, $\rank(Y) = r$ and the system matrix $A$ is diagonalizable by $r$ non-zero distinct eigenvalues.
Therefore, the assumptions of \cref{2:Theorem_Reconst_System_matrix_A} are satisfied such that there exist scaling factors $\alpha_1,\dotso,\alpha_r \in \mathbb{C}$ and an error $q_0 \in \mathbb{C}^n$ such that
\begin{equation*} 
x_0 = \sum_{j=1}^{r} \alpha_j v_j + q_0,
\qquad
x_k = \sum_{j=1}^{r} \lambda_j^k \alpha_j v_j,
\end{equation*}
for $k = 1,\dots,m$. Using the notation of \cref{2:Theorem_Reconst_System_matrix_A}, we obtain the equation
\begin{equation*}
Y = \Theta K \Lambda \Vand(\lambda_1,\dots,\lambda_r;1,m),
\end{equation*}
where $K=\text{diag}(\alpha_1,\dotso,\alpha_m)$ is the scaling matrix.
Since the DMD modes (which are eigenvectors of the system matrix) are related to distinct eigenvalues, they are linearly independent and, consequently, the following calculation 
\begin{align*}
a = \Lambda^{-1} \Theta^+ x_1 &= \Lambda^{-1} \Theta^+ Y e_1 \\
&= \Lambda^{-1} \Theta^+ \Theta K \Lambda \Vand(\lambda_1,\dots,\lambda_r;1,m)  e_1 \\
&= K \Vand(\lambda_1,\dots,\lambda_r;1,m) e_1 \\
&= (\alpha_1,\dots,\alpha_m)^T
\end{align*}
shows that the scaling factors $\alpha_k$ coincide with the DMD amplitude $a_k$. 
Hence, the two identities are proven. For the additional statement, we consider two cases:

1. case: $r = m$. Consider $Y = XC_c + qe_m^T$ and rearrange this equation by using the inverse of $C_c$ (which exists since $r=m$):
\begin{equation*}
X = (Y - q e_m^T) C_c^{-1} = Y C_c^{-1} - q e_m^T C_c^{-1}.
\end{equation*}
By \cref{4:Lemma_CS_similar} and relation $C_c = (\Sigma V^*) S (\Sigma V^*)^{-1}$, we obtain the identity
\begin{equation*}
Y C_c^{-1}
= Y V \Sigma^{-1} S^{-1} \Sigma V^*
= \Theta \Lambda W S^{-1} \Sigma V^*
= \Theta W \Sigma V^* 
= \Theta W \Sigma V^* M^{-1} M
= \Theta K M,
\end{equation*}
and, therefore, it follows
\begin{equation*}
x_0 = X e_1 
= (Y C^{-1} - q e_m^T C_c^{-1})e_1
= \Theta K M e_1 - q e_m^T C_c^{-1}e_1
= \sum_{j=1}^m {a_j \vartheta_j} - \langle e_m, C_c^{-1} e_1 \rangle \cdot q.
\end{equation*}
Now, we examine the remaining term $\langle e_m , C_c^{-1} e_1 \rangle \cdot q$.
To this end, we represent the companion matrix $C_c$ by its eigenvectors and eigenvalues
\begin{align*}
e_m^T C^{-1} e_1 
&= e_m^T \Vand(\lambda_1,\dots,\lambda_m)^{-1} \Lambda^{-1} \Vand(\lambda_1,\dots,\lambda_m) e_1 \\
&= ((\Vand(\lambda_1,\dots,\lambda_m)^T)^{-1} e_m)^T \Lambda^{-1} \Vand(\lambda_1,\dots,\lambda_m) e_1.
\end{align*} 
The inverse of the transposed Vandermonde matrix $\Vand(\lambda_1,\dots,\lambda_m)^T$ is  given by a LU-decomposition~\cite{moya:2012:inverse}, i.e., there exist a lower triangle matrix $L$ and a upper triangle matrix $U$ with
\begin{equation*}
(\Vand(\lambda_1,\dots\lambda_m)^T)^{-1} = (LU)^{-1} = U^{-1} L^{-1}.
\end{equation*}
More precisely, these matrices are given by
\begin{equation*}
U^{-1}_{i,j} = \begin{cases}
0 	&i>j\\
\prod_{\substack{k=1 \\ k \neq i}}^{j} \frac{1}{\lambda_i - \lambda_k}	&i\leq j
\end{cases}
\end{equation*}
and
\begin{equation*}
L^{-1}_{i,j} = \begin{cases}
0		& i < j \\
1		& i = j \\
L_{i-1,j-1} - L_{i-1,j} \cdot \lambda_{i-1}		& i = 2,\dots,m, \quad j = 2,\dots,i-1.
\end{cases}
\end{equation*}
Thus we obtain
\begin{equation*}
(\Vand(\lambda_1,\dots,\lambda_m)^T)^{-1} e_m 
= U^{-1} L^{-1} e_m 
= U^{-1} e_m
= \begin{pmatrix}
\prod_{\substack{k=1 \\ k \neq i}}^{m} \frac{1}{\lambda_1 - \lambda_k} \\
\vdots \\
\prod_{\substack{k=1 \\ k \neq i}}^{m} \frac{1}{\lambda_m - \lambda_k}
\end{pmatrix}
\end{equation*}
and finally
\begin{align*}
e_m^T C^{-1} e_1
&= ((\Vand(\lambda_1,\dots,\lambda_m^T)^{-1} e_m)^T \Lambda^{-1} \Vand(\lambda_1,\dots,\lambda_m e_1 \\
&= \frac{1}{\lambda_1} \prod_{\substack{k=1 \\ k \neq i}}^{m} \frac{1}{\lambda_1 - \lambda_k} + \dots +
\frac{1}{\lambda_m} \prod_{\substack{k=1 \\ k \neq i}}^{m} \frac{1}{\lambda_m - \lambda_k}.
\end{align*}
This factor equals the error scaling $a_0$, which completes the proof for the first case.

2. case: $r \neq m$. Hence, $\ker(X)$ is non-trivial, i.e., $\ker(X) \neq \{0\}$. By \cref{5:Lemma_kerXkerY_kerXneq0}, we obtain $\im(X) \subseteq \im(Y)$ and, therefore, it holds
\begin{equation*}
\im(X) = \im(Y) = \im(A) = \langle \vartheta_1,\dots,\vartheta_r \rangle,
\end{equation*}
because $\im(A) \subseteq \im(Y)$.
Hence, $x_0$ is in the span of $\vartheta_1,\dots,\vartheta_r$ and therefore $q_0 \in \im(A)$, especially.
However, by the proof of \cref{2:Theorem_Reconst_System_matrix_A}, we obtain that the vector $q_0 \in \ker(A)$. 
Consequently, we obtain $q_0 = 0$.
\end{proof}

\cref{5:Theorem_Reconst-EXDMD} clearly demonstrates the functionality of EXDMD.
Via a reduced singular value decomposition, the dynamicly relevant properties of the system matrix were extracted, i.e., the eigenvectors to non-zero eigenvalues.
As a result, the eigenvectors (or DMD modes) do not not generate a basis anymore and consequently the first snapshot $x_0$ will be reconstructed with an error.
However, \cref{5:Theorem_Reconst-EXDMD} gives us an exact representation of the error.
The subsequent corollary illustrates the connection between the DMD amplitudes and the coefficients of the first snapshot in the eigenvector basis of the system matrix.

\begin{corollary}
Let the system matrix $A$ as well as DMD eigenvalues $\lambda_1,\dots,\lambda_r$, DMD modes $\vartheta_1,\dots,\vartheta_r$, DMD amplitudes $a_1,\dots,a_r$ be given by \cref{5:Algorithmus_EXDMD} to the data $x_0, \allowbreak\dotso, \allowbreak x_m \allowbreak \in \mathbb{C}^n$. In addition, let the eigenvalues $\lambda_1,\dots,\lambda_r,\lambda_{r+1},\dots,\lambda_n$ and eigenvectors $\vartheta_1, \allowbreak \dots, \allowbreak\vartheta_r, \allowbreak v_{r+1}, \allowbreak\dots, \allowbreak v_n$ of the system matrix as well as the coefficient vector $b = (b_1,\dots,b_n)^T = W^{-1} x_0$ with $W = (\vartheta_1,\dots,\vartheta_r,v_{r+1},\dots,v_m)$ be given. If $\ker(X) \subseteq \ker(Y)$, $\rank(Y) = r$, and the (DMD) eigenvalues $\lambda_1,\dots,\lambda_r$ are distinct, then the DMD amplitudes coincide with the coefficients:
\begin{equation*}
a_j = b_j
\end{equation*}
for $j=1,\dots,r$.
\end{corollary}
\begin{proof} The assumptions of \cref{5:Theorem_Reconst-EXDMD} and \cref{2:Theorem_Reconst_System_matrix_A} are satisfied, because the system matrix $A$ is diagonalizable by the $r$ distinct eigenvalues. Hence we receive
\begin{equation*}
\sum_{j=1}^r \lambda_j a_j \vartheta_j 
= x_1 
= \sum_{j=1}^r \lambda_j b_j \vartheta_j.
\end{equation*}
Since the eigenvectors (to distinct eigenvalues) are linear independent and the eigenvalues are non-zero, the coefficients have to match.
\end{proof}

\section{Conclusion}
A comprehensive theoretical analysis of Dynamic Mode Decomposition has been developed that clarify the connection between different variants of DMD (CDMD, SDMD, and EXDMD) and demonstrates several features of them.
One of these features is the reconstruction property, which was proven for all variants and the system matrix as well.
To this end, different scaling factors were used and new ones introduced to ensure this property.
Especially for EXDMD, it was shown that under appropriate conditions the algorithm calculates the dynamically relevant, high-dimensional structures of the system matrix with the help of low-dimensional, spectral-theoretical techniques.
The new findings facilitate the application with DMD since precise reconstructions are obtained which lead to a clearer decomposition of the data into DMD eigenvalues, modes and amplitudes.

\section*{Acknowledgment}
The first author would like to thank Rainer Nagel for ideas, helpful suggestions and inspiring discussions.
Moreover, he expresses his gratitude to the entire working group functional analysis (AGFA) from the University of T{\"u}bingen for the numerous discussions and the support.

This work is partly supported by ``Kooperatives Promotionskolleg Digital Media'' at Hochschule der Medien and the University of Stuttgart.

\bibliographystyle{plain}
\bibliography{bib}

\begin{thebibliography}{10}

\bibitem{chen:2012:variants_of_DMD}
Kevin~K. Chen, Jonathan~H. Tu, and Clarence~W. Rowley.
\newblock {Variants of dynamic mode decomposition: boundary condition, Koopman,
  and Fourier analyses}.
\newblock {\em Journal of Nonlinear Science}, 22(6):887--915, 2012.

\bibitem{drmavc:2018:data_driven_koopman_spectral_analysis_in_vandermonde_cauchy_form}
Zlatko Drma{\v{c}}, Igor Mezi{\'c}, and Ryan Mohr.
\newblock {Data driven Koopman spectral analysis in Vandermonde-Cauchy form via
  the DFT: numerical method and theoretical insights}.
\newblock {\em arXiv preprint arXiv:1808.09557}, 2018.

\bibitem{jovanovic:2014:sparsity_promoting_dmd}
Mihailo~R. Jovanovi{\'c}, Peter~J. Schmid, and Joseph~W. Nichols.
\newblock Sparsity-promoting dynamic mode decomposition.
\newblock {\em Physics of Fluids}, 26(2):024103, 2014.

\bibitem{moore:1920:moore_penrose}
Eliakim~H. Moore.
\newblock On the reciprocal of the general algebraic matrix.
\newblock {\em Bulletin of the American Mathematical Society}, 26:394--395,
  1920.

\bibitem{moya:2012:inverse}
Hector Moya-Cessa and Francisco Soto-Eguibar.
\newblock {Inverse of the Vandermonde and Vandermonde confluent matrices}.
\newblock {\em Applied Mathematics \& Information Sciences}, 5:361, 2011.

\bibitem{rowley:2009:spectral_nonlinear_flow}
Clarence~W. Rowley, Igor Mezic, Shervin Bagheri, Philipp Schlatter, and Dan~S.
  Henningson.
\newblock Spectral analysis of nonlinear flows.
\newblock {\em Journal of Fluid Mechanics}, 641:115--127, 2009.

\bibitem{schmid:2010:dmd_numerical_data}
Peter~J. Schmid.
\newblock Dynamic mode decomposition of numerical and experimental data.
\newblock {\em Journal of Fluid Mechanics}, 656:5--28, 2010.

\bibitem{schmid:2011:applications_of_dmd}
Peter.~J. Schmid, Larry Li, Matthew~P. Juniper, and Oliver Pust.
\newblock Applications of the dynamic mode decomposition.
\newblock {\em Theoretical and Computational Fluid Dynamics}, 25:249--259,
  2011.

\bibitem{schmid:2008:APS}
Peter~J Schmid and J{\"o}rn Sesterhenn.
\newblock Dynamic mode decomposition of numerical and experimental data.
\newblock In {\em APS Division of Fluid Dynamics Meeting Abstracts}, page
  MR.007, November 2008.

\bibitem{tu:2014:on_dmd_theorey_and_app}
Jonathan~H. Tu, Clarence~W. Rowley, Dirk~M. Luchtenburg, Steven~L. Brunton, and
  J.~Nathan Kutz.
\newblock On dynamic mode decomposition: Theory and applications.
\newblock {\em Journal of Computational Dynamics}, 1:391--421, 2014.

\end{thebibliography}

\end{document}